\documentclass{article}

\usepackage{arxiv}

\usepackage[utf8]{inputenc} 
\usepackage[T1]{fontenc}    
\usepackage{hyperref}       
\usepackage{url}            
\usepackage{booktabs}       
\usepackage{amsfonts}       
\usepackage{nicefrac}       
\usepackage{microtype}      
\usepackage{doi}
\usepackage{enumitem}
\usepackage{style}
\usepackage{graphicx}
\usepackage{tikz}
\usepackage{doi}
\usepackage[normalem]{ulem}

\title{The Zipped Finite Element Method on polyhedral meshes}

\author{ \href{https://orcid.org/0000-0001-8642-4258}{\includegraphics[scale=0.06]{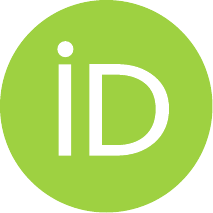}\hspace{1mm}Stefano~Berrone} \\
	Dipartimento di Scienze Matematiche
	``G. L. Lagrange'', \\Politecnico di Torino, \\
    Corso Duca degli Abruzzi, 24, 10129, Turin, Italy\\
	\texttt{stefano.berrone@polito.it} \\
	\And
	\href{https://orcid.org/0000-0002-8540-3639}{\includegraphics[scale=0.06]{orcid.pdf}\hspace{1mm}Gioana~Teora} \\
	Dipartimento di Scienze Matematiche
	``G. L. Lagrange'',\\ Politecnico di Torino, \\
    Corso Duca degli Abruzzi, 24, 10129, Turin, Italy \\
	\texttt{gioana.teora@polito.it} 
}

\renewcommand{\shorttitle}{The Zipped Finite Element Method on polyhedral meshes}
\hypersetup{
pdftitle={},
pdfsubject={},
pdfauthor={Stefano~Berrone, Gioana~Teora},
pdfkeywords={},
}

\begin{document}
\maketitle

\begin{abstract}
We extend the high-order Zipped Finite Element Method to three-dimensional polyhedral elements. The construction relies on a local sub-tetrahedralization obtained by triangulating the polygonal faces and connecting the resulting vertices to a suitable interior point. High-order shape functions are constructed as linear combinations of standard finite element basis functions, with coefficients chosen to ensure exact polynomial reproduction on the faces and in the element interior while preserving global $\con{0}{}$-conformity. We further introduce a QR-based strategy as a systematic alternative to the heuristic procedure that was adopted for the two-dimensional case. The QR-strategy is more systematic and robust than the heuristic strategy, while being independent of the spatial dimension and avoiding geometric alignment tests and problem-dependent tolerances. The resulting method is formulated for a diffusion-reaction problem and numerically assessed through three-dimensional convergence tests, confirming the expected optimal polynomial order of accuracy.
\end{abstract}


\section{Introduction}

Over the last two decades, the numerical solution of partial differential equations on domains with complex geometry has motivated a substantial shift away from the classical finite element paradigm, in which elements are restricted to simplices, quadrilaterals, or hexahedra. The Finite Element Method (FEM), in its standard formulation, offers a mature and well-understood theoretical framework \cite{ciarlet2002finite}, but its reliance on a limited set of reference element shapes can impose significant constraints when dealing with complex geometries, fractures, heterogeneous material interfaces, or adaptive refinement and coarsening strategies \cite{Benedetto2016, Vicini2024, GrappeinTeora2025}. General polytopal (i.e. polygonal and polyhedral) meshes remove this restriction, allowing elements with an arbitrary number of edges and faces, non-convex shapes, and hanging nodes. This flexibility has driven the development of an entire family of polytopal methods, among which the Virtual Element Method (VEM) \cite{LBe13}, the Hybrid High-Order (HHO) method \cite{DiPietroErnLemaire2014}, polygonal discontinuous Galerkin schemes \cite{Antonietti2021}, and, closer to the classical FEM philosophy, polygonal/polyhedral finite elements built on generalized barycentric coordinates \cite{Wachspress1975, FLOATER2003, Sukumar2004}.

A central and comparatively less resolved question within this landscape is how to write in a closed-form a set of basis functions able to reproduce polynomials of high order while preserving global $\con{0}{}$-conformity, particularly in three dimensions. In the two-dimensional setting, the problem is by now reasonably well understood: rational or generalized barycentric coordinates, such as Wachspress, mean-value, maximum-entropy, harmonic and Laplace coordinates \cite{floater2015generalized}, provide conforming elements on convex and, in some constructions, non-convex polygons, while serendipity-type reductions of the Degrees of Freedom (DOFs) to match those of a standard FEM on triangles have been derived explicitly \cite{SUKUMAR201327}. Within the Virtual Element Method itself, a further family of recent variants has emerged with the specific aim of computing the virtual basis functions explicitly rather than only through their projection and stabilization, thereby recovering pointwise-evaluable basis functions in two-dimensions. The reduced-basis Virtual Element Method (rbVEM) constructs cheap approximations of the virtual basis functions through a reduced-basis strategy, used both to design the stabilization term and to post-process the solution \cite{Credali2024}. The Lightning Virtual Element Method (LVEM) approximates the virtual basis function with rational functions \cite{TrezziZerbinati2024, Gopal2019}, whereas the Neural Approximated Virtual Element Method (NAVEM) solves the local Laplace problem that locally defines virtual functions by a neural network \cite{PintoreTeora2025, PintoreTeora2026}.

Extending this type of explicit, conforming construction to three-dimensional polyhedra is considerably more delicate. Several difficulties compound in 3D that have no direct two-dimensional analogue. First, a polyhedral element is bounded by faces that are themselves polygons of arbitrary shape. Any conforming high-order construction must therefore guarantee that the traces of the basis functions on shared faces coincide exactly between neighbouring elements, a compatibility condition that becomes more demanding as the number and irregularity of faces grows. Second, no general closed-form generalization of rational barycentric coordinates preserving linear precision is available for arbitrary, possibly non-convex, polyhedra, unlike the relatively mature two-dimensional theory. Third, the serendipity degree-of-freedom reduction, which in 2D can be characterized rather explicitly \cite{SUKUMAR201327}, becomes substantially more involved in 3D, where degrees of freedom are associated with vertices, edges, and faces. Da Veiga et al. \cite{DaVeigaBrezzi2018} address this issue in the context of the Virtual Element Method by restricting the serendipity reduction to the two-dimensional faces of a polyhedral partition and applying static condensation for removing bulk DOFs, because a fully three-dimensional reduction is considerably less straightforward.

Bishop \cite{Bishop2014} first introduces a first-order continuous-Galerkin scheme for general polytopes, with an arbitrary number of vertices, edges, and faces, possibly non-planar faces, and star-convexity with respect to the vertex-averaged centroid as the only shape restriction, in which conforming shape functions are obtained as a weighted combination of FEM functions defined over a sub-tetrahedralization of the element. This construction guarantees exact partition of unity, linear completeness, and inter-element conformity by design, but only the lowest-order case was addressed. The recently proposed Zipped Finite Element Method (Z-FEM) \cite{NevaTeora2025} follows a related line of work, extending this construction to build high-order shape functions on star-shaped polygons as linear combinations of standard finite element bases defined on a local sub-triangulation, and subsequently ``zipping'' the resulting space down to a serendipity-number of degrees of freedom, so that global $\con{0}{}$-conformity is exactly preserved by construction and the polynomial reproducibility can be exactly enforced. An independent and simultaneous work by Sellmann et al. \cite{SellmannWriggers2026} introduced the two-dimensional Projection Enhanced Triangular Finite Element Method (PET-FEM), which is based on the same idea as the Z-FEM, but it employs different types of degrees of freedom and sub-triangulations for the construction of high-order shape functions.

In this work, the extension of the Z-FEM to three dimensions is established. In 3D, the sub-tetrahedralization of the polyhedral element is obtained by first sub-triangulating the faces, according to the two-dimensional procedure developed in \cite{NevaTeora2025}, and then connecting the derived vertices to a suitably chosen interior point located within the kernel of the element. To guarantee $\con{0}{}$-conformity, the coefficients of the linear combinations are determined by solving a local optimization problem on each face, which guarantees reproducing the two-dimensional polynomials over each face of the polyhedron, and then polynomial reproducibility is enforced also in the bulk of the element. We remember that this local optimization problem can be efficiently decomposed into smaller sub-problems that consist of a set of linear systems characterized by the same factorizable coefficient matrix. Thus, the local space is designed so that polynomials of degree up to the method order can be exactly reproduced. This property is fundamental, as it allows the method to inherit the theoretical framework of the FEM, ensuring both well-posedness of the discrete problem and optimal a priori error estimates, but over arbitrary shape elements. The number of degrees of freedom is selected according to a lazy-serendipity criterion \cite{DaVeigaBrezzi2018} for faces, following \cite{NevaTeora2025}, with a straightforward extension to three dimensions. The heuristic procedure proposed in \cite{NevaTeora2025} selects, among the nodal coordinates of the underlying FEM space, the internal nodal coordinates that improve homogeneity and minimize alignment. While computationally inexpensive, this index-based approach for determining DOFs is heuristic and may fail. Here, we therefore introduce an alternative QR-strategy, which is more computationally demanding but more systematic and robust. Like the heuristic approach, it avoids the geometric complexity of classical serendipity constructions, requiring neither the identification of aligned entities nor problem-dependent geometric tolerances. The QR-strategy is also dimension-independent and can therefore be applied in arbitrary spatial dimensions

The outline of the paper is as follows. In Section \ref{sec:notationsmodel}, we introduce some preliminary notations and the PDE model problem. In Section \ref{sec:zippedfem}, we introduce the high-order Zipped Finite Element Space in three dimensions and define the dimension of the space. Section \ref{sec:shapefunctions} details the construction of the zipped basis functions and the related choice of degrees of freedom, whereas Section \ref{sec:discreteproblem} defines the discrete problem and the theoretical convergence properties of the method. Finally, Section \ref{sec:numericalexperiments} proposes some numerical experiments aimed at testing convergence findings, while Section \ref{sec:conclusion} draws the conclusion.

\section{Notations and the Model Problem}
\label{sec:notationsmodel}

Throughout the paper, the usual notation for Sobolev spaces is adopted. Given a generic polytope $\mathcal{O} \subset \R^{d}$, with $d=1,\dots,3$, we denote by $\scal[\mathcal{O}]{\cdot}{\cdot}$ the $\leb{2}{\mathcal{O}}$-scalar product. Furthermore, for each for $s \geq 0$, $\norm[\sob{s}{\mathcal{O}}]{\cdot}$ and $\seminorm[\sob{s}{\mathcal{O}}]{\cdot}$ denote the norm and seminorm of functions in $\sob{s}{\mathcal{O}}$, respectively.

Let us introduce the space $\Poly[d]{k}{\mathcal{O}}$ of $d$-dimensional polynomials of degree up to $k \in \N$ defined over $\mathcal{O}$, whose dimension is $n_k^d := \frac{(k+1)\dots(k+d)}{d!}$. We adopt the standard convection $\Poly[s]{-m}{\mathcal{O}} = \emptyset$ and $n^d_{-m} = 0$ for all $m \in \N$. Let us introduce the set of $d$-dimensional (scaled) monomials of degree up to $k$ on $\mathcal{O}$:
\begin{equation}
    \M[d]{k}{\mathcal{O}} = \left\{m_{\alpha}^d(\xx) = \left(\frac{\xx - \xx_{\mathcal{O}}}{h_{\mathcal{O}}}\right)^{\bm{\alpha}}: \bm{\alpha} = \mathbf{idx}(\alpha)\ \, \forall \alpha = 1,\dots,n_k^d \right\},
\end{equation}
where $\xx_{\mathcal{O}}$ represents the centroid of $\mathcal{O}$, $\mathbf{idx} : \{1,\dots, n_k^d\} \to \N^d$ is the natural function associating each natural number $\alpha \in \{1,\dots, n_k^d\}$ to a multi-index $\bm{\alpha} \in \N^d$, and the symbol $h_{\mathcal{O}} \coloneq \max_{\xx,\yy \in \mathcal{O}} \norm[2]{\xx-\yy}$ denotes the diameter of $\mathcal{O}$.

Let us now focus on the three-dimensional setting, i.e. $d=3$. Let $\Omega \subset \R^3$ be a bounded polyhedral domain with Lipschitz boundary $\Gamma \coloneq \partial \Omega$. We consider the following diffusion-reaction problem with homogeneous Dirichlet boundary conditions:
\begin{equation}
    \begin{cases}
        -\nabla \cdot (\D \nabla u) + \gamma u = b & \text{in } \Omega,\\
        u = 0 &\text{on } \Gamma,\\
    \end{cases}
    \label{eq:modelproblem}
\end{equation}
where $\gamma \in \leb{\infty}{\Omega}$, $\gamma(\xx) \geq 0$ for all $\xx \in \Omega$, is the reaction coefficient, and $b \in \leb{2}{\Omega}$ is the source term.
Moreover, $\D \in \vleb{\infty}{\Omega}{3 \times 3}$ represents the diffusion tensor, which is uniformly symmetric positive definite over $\Omega$, i.e. there exist two constants $0 < \alpha_1 \leq \alpha_2$ such that 
\begin{equation*}
    \alpha_1 \norm[2]{\bm{\epsilon}} \leq \bm{\epsilon}^T \D(\xx) \bm{\epsilon} \leq \alpha_2 \norm[2]{\bm{\epsilon}} \qquad \forall \xx \in \Omega,\quad \forall\bm{\epsilon} \in \R^3,
\end{equation*}
where $\norm[2]{}$ denotes the euclidean norm.

The variational formulation of the Problem \eqref{eq:modelproblem} reads as: \textit{Find $u \in \VP := \sob[0]{1}{\Omega}$ such that}
\begin{equation}
    \mathcal{B}(u, v) = \scal[\Omega]{b}{v} \quad \forall v \in \VP,
    \label{eq:varproblem}
\end{equation}
where $\bilin{}{} : \VP \times \VP \to \R$ is the following symmetric, continuous, and coercive bilinear form
\begin{equation}
\mathcal{B}(u,v) = \scal[\Omega]{\D \nabla u}{\nabla v} + \scal[\Omega]{\gamma u}{v} \quad \forall u,v \in \VP.
\label{eq:formabilincontinua}
\end{equation}

\section{The Zipped Finite Element Space}
\label{sec:zippedfem}

Let $\Th$ be a discretization of the domain $\Omega$ into non-overlapping polyhedral elements $E$. The symbols $\Nv[E]$, $\Ne[E]$, and $\Nf[E]$ denote the number of vertices, edges, and faces of the polyhedron $E$, respectively. We denote by $\Eh[E]$ and $\Fh[E]$ the set of edges and faces of $E$, respectively. Moreover, we define the global sets $\Eh \coloneq \cup_{E \in \Th} \Eh[E]$ and $\Fh \coloneq \cup_{E \in \Th} \Fh[E]$, and we set $h \coloneq \max_{E \in \Th} h_E$. Finally, let $\{\vv^{\mathcal{O}}_{ j}\}_{j = 1}^{\Nv[\mathcal{O}]}$ denotes the sets of vertices of the polytope $\mathcal{O}$.

We assume $\Th$ satisfies the following standard mesh assumptions \cite{Ahmad2013}.
\begin{assumptions}[Mesh assumptions]\label{ass:meshassumption}
There exists a positive real number $\rho \in (0, 1)$, independent of $h$, such that
\begin{itemize}
    \item for each element $E\in\Th$, each face $f$ of $E$ and each edge $e$ of $f$, it holds $h_e \geq \rho h_f \geq \rho^2 h_E$;
    \item each element $E\in\Th$ is star-shaped with respect to all the points of a sphere of radius $\geq \rho h_E$;
    \item each face $f \in \Fh$ is star-shaped with respect to all points in a disk of radius $\geq \rho h_f$.
\end{itemize}
\end{assumptions}

Let $\xx_{E}$ and $\xx_f$ denote the centers of the largest ball contained in $E$ and the largest disk contained in $f$, respectively, such that $E$ is star-shaped with respect to every point of the ball centered at $\xx_E$, while $f$ is star-shaped with respect to every point of the disk centered at $\xx_f$. The point $\xx_E$ can be found as the solution of the following linear optimization problem \cite{NevaTeora2025, Calafiore2014}: 
\begin{equation}
\begin{aligned}
    \displaystyle\max_{(\xx_{E}, \varrho_{E}) \in \R^{4}}\quad &\varrho_{E}\\
    \text{such that} \quad& \nn_f \cdot \xx_E + \varrho_E \leq \mu_f  \quad \forall f \in \Fh[E]\\
    & 0 < \varrho_E \leq h_E.
\end{aligned}
\label{eq:opt:centerstarshaped}
\end{equation}
where $\varrho_{E}$ represents the radius of the sphere centered at $\xx_E$, $\nn_f$ denotes the unit normal vector to the face $f \in \Fh[E]$ pointing outward from the polyhedron $E$, whereas the scalar coefficient $\mu_f$ is defined as $\mu_f \coloneq \nn_f \cdot \vv^f_{1}$. The point $\xx_f$, for each face $f \in \Fh[E]$, can be found by solving an analogous problem.

Let us define the sub-triangulation $\TE[f] = \{T_j^f\}_{j=1}^{\Nv[f]}$ of a face $f$ by connecting each pair of consecutive vertices of $f$ with the interior point $\xx_f$, thereby forming the triangles
\begin{equation*}
    T_j^f = \{\xx_f, \vv_j^f,\vv_{j+1}^f\} \quad  j=1,\dots, \Nv[f],
\end{equation*}
where we adopt the convection $\vv_{\Nv[f]+1}^f = \vv_1^f$. The sub-tetrahedralization $\TE[E]$ of the polyhedron $E$ is then obtained by connecting the vertices of each triangle $T_j^f$ to the interior point $\xx_E$. In particular, we define the tetrahedra as
\begin{equation*}
    T_{j,f}^E = \{\xx_f, \vv_j^f,\vv_{j+1}^f, \xx_E\} \quad  j=1,\dots, \Nv[f] \quad \forall f \in \Fh[E].
\end{equation*}

\begin{figure}[!ht]
    \centering
    \begin{subfigure}{0.4\textwidth}
        \includegraphics[width=1\linewidth]{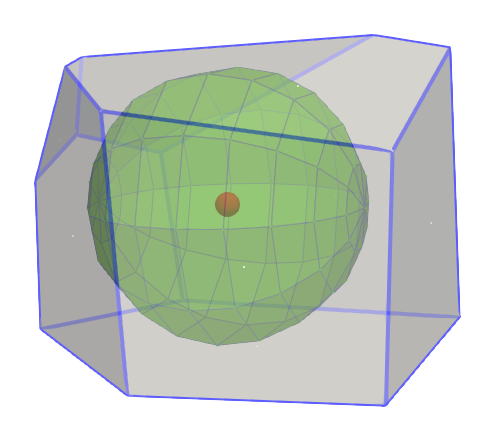}
        \caption{}
    \end{subfigure}
    \begin{subfigure}{0.4\textwidth}
        \includegraphics[width=1\linewidth]{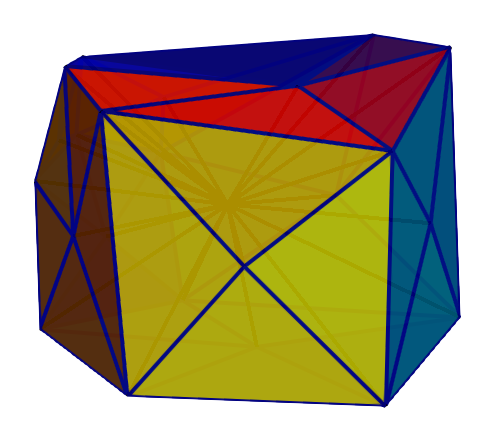}
        \caption{}
    \end{subfigure}
    \caption{A generic polyhedron. Left: In green, the insphere centered at $\xx_E$, highlighted in red. Right: Sub-tetrahedralization.}
    \label{fig:element-tetra}
\end{figure}
Figure \ref{fig:element-tetra} shows the computed insphere, its center $\xx_E$, and the derived sub-tetrahedralization for a generic polyhedron.

The standard finite element space of order $k \geq 1$ over the sub-tetrahedralization $\TE[E]$ is defined as
\begin{equation}
    \mathbb{V}_k(E; \TE[E]) := \{v \in \con{0}{\overline{E}} \cap \sob{1}{E} : v_{|T} \in \Poly{k}{T}\ \ \forall\, T \in \TE[E] \}.
    \label{eq:fintielementlocalspace}
\end{equation}
We observe that $\Poly{k}{E} \subset \mathbb{V}_k(E; \TE[E])$. Moreover, the dimension of this space is defined as $N_E^{\Psi,k} \coloneq \dim \mathbb{V}_k(E; \TE[E])$ and it is given by
\begin{equation*}
    N_E^{\Psi,k} = \Nv[E] + \Nf[E] + 1 + \left(\Ne[E] + \Nv[E] + \Nf[E] + \sum_{f \in \Fh[E]} \Nv[f]\right) (k - 1)  + \left(\Ne[E] + 2 \sum_{f \in \Fh[E]} \Nv[f] \right) n_{k-3}^2 + \left(\sum_{f \in \Fh[E]} \Nv[f] \right)n_{k-4}^3.
\end{equation*} 

The zipped finite element method consists of \emph{compressing}
the space $\mathbb{V}_k(E; \TE[E])$ down to a new space $\VPh[E]{k}$ of dimension
\begin{equation}
\Ndof[E] \coloneq \; \Nv[E] + \Ne[E] (k - 1) + \Nf[E] n_{k-3}^2 + n_{k-4}^3.
\label{eq:Ndof}
\end{equation}
This dimension reflects the choice of degrees of freedom in the Z-FEM space. Specifically, there is one degree of freedom associated with each vertex, $k-1$ degrees of freedom associated with each edge, $n_{k-3}^2$ internal degrees of freedom associated with each face, and $n_{k-4}^3$ internal degrees of freedom associated with the polyhedron. The vertex and edge degrees of freedom uniquely determine a polynomial of degree $k$ along each edge of the polyhedron. The additional $n_{k-3}^2$ face-internal degrees of freedom and $n_{k-4}^3$ polyhedron-internal degrees of freedom then ensure the exact reproduction of polynomials of degree up to $k$ on each face $f$ and in $E$, respectively. This choice of degrees of freedom is consistent with the so-called \textit{lazy choice} adopted for serendipity elements \cite{DaVeigaBrezzi2016, DaVeigaDassi2017}.

Let $\finenodes[E] = {1,\dots,N_E^{\Psi,k}}$ denote the index set of the Lagrange nodes associated with $\mathbb{V}_k(E;\TE[E])$, which we refer to as the \textit{fine nodes}. We partition this set into the following subsets:
\begin{itemize}
    \item $\mathcal{I}_{v,e}^{\mathcal{C}}$ indexing the vertex and edge nodes;
    \item $\mathcal{I}_f$ indexing the nodes internal to the face $f\in\Fh[E]$; 
    \item $\mathcal{I}_E$ indexing the nodes internal to the polyhedron $E$.
\end{itemize}

For each face $f\in\Fh[E]$, we partition the set $\mathcal{I}_f$ into two subsets. The first, denoted by $\mathcal{I}_f^{\mathcal{C}}$, contains $n_{k-3}^2$ suitably chosen indices, while the second consists of the remaining indices, denoted by $\virtualnodes[f]$, associated with the nodal coordinates $\{\pp_j^f\}_{j \in \virtualnodes[f]}$. These nodal coordinates are referred to as the face virtual nodes. The set $\coarsenodes[f]$, defined as the union of the face internal coarse nodes $\mathcal{I}_f^{\mathcal{C}}$ and the indices associated with the vertex and edge nodes belonging to $f$, is instead referred to as the set of indices of the face coarse nodes. The corresponding nodal coordinates are denoted by $\{\xx_i^f\}_{i \in \coarsenodes[f]}$. 

An analogous construction is introduced for the polyhedron $E$. In this case, we partition the set $\mathcal{I}_E$ into $\mathcal{I}_E^{\mathcal{C}}$, containing $n_{k-4}^3$ suitably chosen indices, and $\virtualnodes[E] = \mathcal{I}_E \setminus \mathcal{I}_E^{\mathcal{C}}$, with the corresponding nodal coordinates denoted by $\{\pp_j^E\}_{j\in\virtualnodes[E]}$. These are referred to as the polyhedron virtual nodes. Similarly, we denote by $\coarsenodes[E]$ the union of vertex and edge nodes, face internal coarse nodes, and the bulk coarse nodes, i.e. we define $\coarsenodes[E] \coloneq \mathcal{I}_{v,e}^{\mathcal{C}} \cup \left(\bigcup_{f \in \Fh[E]}\mathcal{I}_f^{\mathcal{C}}\right) \cup \mathcal{I}_E^{\mathcal{C}}$. The set $\coarsenodes[E]$ is called the set of coarse nodes, and its corresponding nodal coordinates are denoted by $\{\xx_i^E\}_{i \in \coarsenodes[E]}$. Moreover, without loss of generality, we always consider $\coarsenodes[E] = \{1,\dots, \Ndof[E]\}$.

We just want to underline that $\virtualnodes[E] \cap \virtualnodes[f] = \emptyset$ and $\coarsenodes[E] \cap \coarsenodes[f] = \coarsenodes[f]$, for each $f \in \Fh[E]$.

For each integer $k \geq 1$ and polyhedron $E$, we define the local zipped finite element space as
\begin{align}
    \label{eq:zfem_constraint:fem}\VPh[E]{k} \coloneq \{ v \in \con{0}{\overline{E}} \cap \sob{1}{E}: i) &\ v \in \mathbb{P}_k(T) \quad \forall  T \in \TE[E], \\
    \label{eq:zfem_constraint:face} 
    ii) &\  v(\pp_j^f)= \sum_{i \in \coarsenodes[f]} \omega_{ij}^f v(\xx_i^f) \quad \forall j \in \virtualnodes[f] \quad \forall f \in \Fh[E],\\
    \label{eq:zfem_constraint:internal}  iii) &\  v(\pp_j^E)= \sum_{i \in \coarsenodes[E]} \omega_{ij}^E v(\xx_i^E) \quad \forall j \in \virtualnodes[E] \},
\end{align}
for a given sets of non-zeros virtual weights $\{\{\omega_{ij}^f\}_{i \in \coarsenodes[f], j \in \virtualnodes[f]}\}_{f \in \Eh[f]}$ and $\{\omega_{ij}^E\}_{i \in \coarsenodes[E], j \in \virtualnodes[E]}$. The computation of these weights and the selection of coarse nodes will be detailed in Sections \ref{sec:weights} and \ref{sec:selection_nodes}, respectively.

On $\VPh[E]{k}$, we choose as degrees of freedom for a function $v \in \VPh[E]{k}$ its values at the polyhedron coarse nodes $\{\xx_i\}_{i \in \coarsenodes[E]}$, namely:
\begin{enumerate}[label={\textbf{D.\arabic*}}]
    \item \label{dof:vertices} the values of $v$ at the vertices of $E$;
    \item \label{dof:edge} if $k\geq2$, the values of $v$ in $k-1$ evenly spaced points internal of the edges of $e \in \Eh[E]$; 
    \item \label{dof:face} if $k\geq3$, for each $f \in \Fh[E]$, the values of $v$ at the $n_{k-3}^2$ face coarse nodes related to the indices $\mathcal{I}_f^{\mathcal{C}}$;
    \item \label{dof:internal} if $k\geq4$, the values of $v$ at the $n_{k-4}^3$ coarse nodes related to the indices $\mathcal{I}_E^{\mathcal{C}}$.
\end{enumerate}

The next proposition proves that this set of degrees of freedom uniquely identifies a function $v \in \VPh[E]{k}$.
\begin{proposition}
    The degrees of freedom \ref{dof:vertices}--\ref{dof:internal} are unisolvent for $\VPh[E]{k}$.
\end{proposition}
\begin{proof}
We observe that the number of degrees of freedom $\Ndof[E]$, defined in \eqref{eq:Ndof}, is equal to the dimension of the space $\VPh[E]{k}$. To prove the unisolvence, we need to show that
\begin{equation*}
    v(\xx_i^E) = 0 \quad \forall i\in \coarsenodes[E] \quad\Rightarrow \quad v \equiv 0 \quad \text{in } E \qquad \forall v \in \VPh[E]{k}.
\end{equation*}

The constraints \eqref{eq:zfem_constraint:face} and \eqref{eq:zfem_constraint:internal} implies that
\begin{equation*}
    v(\pp_j^f) = 0 \quad \forall j \in \virtualnodes[f] \quad \forall f \in \Fh[E],
\end{equation*}
and
\begin{equation*}
    v(\pp_j^E) = 0 \quad \forall j \in \virtualnodes[E].
\end{equation*}
Since $\finenodes[f] = \coarsenodes[E] \cup \left(\bigcup_{f \in \Fh[E]} \virtualnodes[f]\right) \cup \virtualnodes[E]$ and the corresponding set of fine nodes is unisolvent for the FEM space $\mathbb{V}(E; \TE[E])$, defined in \eqref{eq:fintielementlocalspace}, this implies that $v \equiv 0$ in $E$.

\end{proof}

\section{High-order shape functions on polyhedrons}\label{sec:shapefunctions}

The goal is to define the set of Z-FEM Lagrange basis functions
$\{\varphi_i\}_{i=1}^{\Ndof[E]}$ for the space $\VPh[E]{k}$, defined in \eqref{eq:zfem_constraint:fem}--\eqref{eq:zfem_constraint:internal}, that satisfy the following properties:
\begin{enumerate}[label=\textbf{P.\arabic*}]
    \item \label{prop:krnocker}The elemental shape functions satisfy the Kronecker-Delta property with respect to the Z-FEM degrees of freedom \ref{dof:vertices}--\ref{dof:internal}, i.e.
    \begin{equation}
        \varphi_n(\xx_i^E) = \delta_{ni} \quad \forall n,i \in \coarsenodes[E].
        \label{eq:lineardofsoperator}
    \end{equation}
    \item\label{prop:continuity} The basis functions are continuous across adjacent elements ($\con{0}{}$-conformity).
    \item \label{prop:polynomial}The set $\VPh[E]{k}$ contains the set of polynomials $\Poly{k}{E}$. 
\end{enumerate}

Let us introduce the set of FEM Lagrange basis functions $\{\Psi_n\}_{n=1}^{N^E_{\Psi,k}}$ for the FEM space $\mathbb{V}_k(E; \TE[E])$ that satisfy the Kronecker-Delta property with respect to the set of fine nodes indexed by $\finenodes[E]$, i.e.
\begin{equation}
    \Psi_m(\bm{\xi}_n^E) = \delta_{nm} \quad \forall n,m \in \finenodes[E],
    \label{eq:fem:krnockerdelta}
\end{equation}
where $\{\bm{\xi}_n^E\}_{n \in \finenodes[E]} \coloneq \{\xx_i^E\}_{i \in \coarsenodes[E]} \cup \left(\bigcup_{f \in \Fh[E]} \{\pp_j^f\}_{\virtualnodes[f]}\right) \cup \{\pp_j^E\}_{\virtualnodes[E]}$ denote the fine nodes.

We can write the Z-FEM shape functions for the polygon $E$ as a weighted combination of finite element basis functions $\{\Psi_n\}_{n=1}^{N^E_{\Psi,k}}$. More precisely, for all point $\forall \xx \in \overline{E}$, we write 
\begin{equation}
    \begin{aligned}
        \varphi_{i}(\xx) &= \Psi_{i}(\xx) + \sum_{f \in \Fh[i]} \sum_{j \in \virtualnodes[f]} \omega_{ij}^f \Psi_{j}(\xx) + \sum_{j \in \virtualnodes[E]} \omega_{ij}^E \Psi_{j}(\xx)\quad &&\forall i \in \mathcal{I}_{v,e}^{\mathcal{C}},\\
        \varphi_{i}(\xx) &= \Psi_{i}(\xx) + \sum_{j \in \virtualnodes[f]} \omega_{ij}^f \Psi_{j}(\xx) + \sum_{j \in \virtualnodes[E]} \omega_{ij}^E \Psi_{j}(\xx)\quad &&\forall i \in \mathcal{I}_{f}^{\mathcal{C}} \quad \forall f \in \Fh[E],\\
        \varphi_{i}(\xx) &= \Psi_{i}(\xx) + \sum_{j \in \virtualnodes[E]} \omega_{ij}^E \Psi_{j}(\xx)\quad &&\forall i \in \mathcal{I}_{E}^{\mathcal{C}},
    \end{aligned}
    \label{eq:basisfunctions}
\end{equation}
where $\Fh[i]$, for all $i \in \mathcal{I}_{v,e}^{\mathcal{C}}$, is the subset of faces belonging to $\Eh[E]$ and adjacent to the node $\xx_i^E$, whereas the virtual weights $\{\{\omega_{ij}^f\}_{i \in \coarsenodes[f], j \in \virtualnodes[f]}\}_{f \in \Eh[f]}$ and $\{\omega_{ij}^E\}_{i \in \coarsenodes[E], j \in \virtualnodes[E]}$ are defined in \eqref{eq:zfem_constraint:face} and \eqref{eq:zfem_constraint:internal}, respectively.

It is immediate to check the following proposition.
\begin{proposition}
    The Z-FEM basis functions defined in \eqref{eq:basisfunctions} satisfy Property \ref{prop:krnocker}, independently of the choice of face and polyhedron coarse nodes and the definition of the weights $\{\{\omega_{ij}^f\}_{i \in \coarsenodes[f], j \in \virtualnodes[f]}\}_{f \in \Eh[f]}$ and $\{\omega_{ij}^E\}_{i \in \coarsenodes[E], j \in \virtualnodes[E]}$.
\end{proposition}

\subsection{The computation of virtual weights}\label{sec:weights}

The computation of the virtual weights aims to gain Property \ref{prop:polynomial}. A polynomial $p \in \Poly[3]{k}{E}$ belongs to $\VPh[E]{k}$ if and only if
\begin{equation*}
    p(\xx) = \sum_{i=1}^{\Ndof[E]} p(\xx_i^E) \varphi_i(\xx) \quad \forall \xx \in \overline{E}.
\end{equation*}

Therefore, to ensure that $p\in\VPh[E]{k}$ for any $p\in\Poly[3]{k}{E}$, we can choose the virtual weights such that the following equations are satisfied 
\begin{align}
    \label{eq:constraint_elementwise:face} m^2_{\alpha} (\pp_{n}^f) &= \sum_{i \in \coarsenodes[f]} m^2_{\alpha}(\xx_i^f) \varphi_i(\pp_n^f) = \sum_{i=1}^{\Ndof[E]} m^2_{\alpha}(\xx_i^f) \omega_{in}^f  \quad \forall \alpha =1,\dots,n_k^2 \quad \forall n \in \virtualnodes[f]\quad \forall f \in \Fh[E] ,\\
    \label{eq:constraint_elementwise:bulk} m^3_{\alpha} (\pp_{n}^E) &= \sum_{i \in \coarsenodes[E]} m^3_{\alpha}(\xx_i^E) \varphi_i(\pp_n^E) = \sum_{i=1}^{\Ndof[E]} m^3_{\alpha}(\xx_i^E) \omega_{in}^E \quad \forall \alpha =1,\dots,n_k^3 \quad \forall n \in \virtualnodes[E].
\end{align}

Let us introduce the following matrices
\begin{align*}
    &\mathbf{D}^f \in \R^{n_k^2 \times \#\coarsenodes[f]}: &&\hspace{-40pt} \mathbf{D}^f = \{\{m^2_{\alpha}(\xx_i^f)\}_{\alpha = 1}^{n^2_k}\}_{i \in \coarsenodes[f]}  &&\hspace{-40pt}\forall f \in \Eh[E],\\
    &\mathbf{D}^E \in \R^{n_k^3 \times \Ndof[E]}: &&\hspace{-40pt} \mathbf{D}^E = \{\{m^3_{\alpha}(\xx_i^E)\}_{\alpha = 1}^{n^3_k}\}_{i = 1}^{\Ndof[E]}  &&\hspace{-40pt}\\
    &\mathbf{V}^f \in \R^{n_k^2 \times \#\virtualnodes[f]}: &&\hspace{-40pt} \mathbf{V}^f = \{\{m^2_{\alpha}(\pp_j^f)\}_{\alpha = 1}^{n^2_k}\}_{j \in \virtualnodes[f]}  &&\hspace{-40pt}\forall f \in \Eh[E],\\
    &\mathbf{V}^E \in \R^{n_k^3 \times \# \virtualnodes[E]}: &&\hspace{-40pt} \mathbf{V}^E = \{\{m^3_{\alpha}(\pp_j^E)\}_{\alpha = 1}^{n^3_k}\}_{j \in \virtualnodes[E]}  &&\hspace{-40pt}\\
    &\mathbf{W}^f \in \R^{\#\coarsenodes[f] \times \#\virtualnodes[f]}: &&\hspace{-40pt} \mathbf{W}^f = \{\{\omega^f_{ij}\}_{i \in \coarsenodes[f]}\}_{j \in \virtualnodes[f]}  &&\hspace{-40pt}\forall f \in \Eh[E],\\
    &\mathbf{W}^E \in \R^{\Ndof[E] \times \# \virtualnodes[E]}: &&\hspace{-40pt} \mathbf{W}^E = \{\{\omega^E_{ij}\}_{i =1}^{\Ndof[E]}\}_{j \in \virtualnodes[E]}  &&\hspace{-40pt}
\end{align*}
the equations \eqref{eq:constraint_elementwise:face} and \eqref{eq:constraint_elementwise:bulk} can be equivalently written in matrix form as
\begin{equation}
    \begin{aligned}
        &\mathbf{D}^f \mathbf{W}^f = \mathbf{V}^f \quad \forall f \in \Fh[E],\\
        &\mathbf{D}^E \mathbf{W}^E = \mathbf{V}^E.
    \end{aligned}
    \label{eq:single_system}
\end{equation}

As remarked in \cite{NevaTeora2025}, if a face $f$ is a triangle, then $n_k^2 = \#\coarsenodes[f]$ and the matrix $\mathbf{D}^f$ is a square matrix. Analogously, if $E$ is a tetrahedron, then $n_k^3 = \Ndof[E]$ and the matrix $\mathbf{D}^E$ is squared. In all the other cases, it holds
\begin{equation*}
    \#\coarsenodes[f] > n_k^2 \quad \text{and} \quad \Ndof[E] > n_k^3,
\end{equation*}
leading to undetermined systems of the form \eqref{eq:single_system}.
To overcome this issue, we proceed as in \cite{Bunge2020, NevaTeora2025} and we determine the weights by solving the following minimization problems
\begin{equation}
    \begin{aligned}
    \min_{\mathbf{W}^f \in \R^{\# \coarsenodes[f] \times \# \virtualnodes[f]}}
        &\sum_{i \in \coarsenodes[f]} \sum_{j \in \virtualnodes[f]} (\omega_{ij}^f)^2\\
    \text{such that}\quad 
   &\mathbf{D}^f \mathbf{W}^f = \mathbf{V}^f,
\end{aligned}
\label{eq:optproblem:face}
\end{equation}
for each $f \in\Fh[E]$, and
\begin{equation}
    \begin{aligned}
    \min_{\mathbf{W}^E \in \R^{\Ndof[E] \times \# \coarsenodes[E]}}
        &\sum_{i \in \coarsenodes[E]} \sum_{j \in \virtualnodes[E]} (\omega_{ij}^E)^2\\
    \text{such that}\quad 
    &\mathbf{D}^E \mathbf{W}^E = \mathbf{V}^E.
\end{aligned}
\label{eq:optproblem:bulk}
\end{equation}
These optimization problems can be solved efficiently and with ease by using Algorithm \ref{alg:kkt}, which requires solving a set of $\# \virtualnodes_{\ast}$ linear systems, all characterized by the same coefficient matrix $\mathbf{\mathbf{D}^{\ast} \left(\mathbf{D}^{\ast}\right)^T}$, for $\ast \in \{f,E\}$.
In \cite{NevaTeora2025}, the following result has been proved.
\begin{proposition}\label{prop:optproblemsolution}
    If the matrices $\mathbf{D}^f$, for all $f \in \Fh[E]$, and the matrix $\mathbf{D}^E$ have full row rank, then problems \eqref{eq:optproblem:face} and \eqref{eq:optproblem:bulk} admit a unique solution, and they can be solved efficiently by using Algorithm \ref{alg:kkt}. Moreover, Property \ref{prop:polynomial} holds, i.e. $\Poly{k}{E} \subseteq \VPh[E]{k}$.
\end{proposition}

\SetKwComment{Comment}{/* }{ */}
\begin{algorithm}
\caption{An algorithm to efficiently solve problems \eqref{eq:optproblem:face} and \eqref{eq:optproblem:bulk}. }\label{alg:kkt}
\KwData{The matrices $\mathbf{D}$ and $\mathbf{V}$.}
\KwResult{The shape function weights $\mathbf{W}$.}
$\mathbf{L} \gets \operatorname{chol}(\mathbf{D}\mathbf{D}^T)$ \Comment*[r]{Compute Cholesky factorization of $\mathbf{D}\mathbf{D}^T = \mathbf{L} \mathbf{L}^T$ once.}
\For{$j = 1,\dots, \# \mathcal{K}$}{ 
    $\mathbf{L} \bm{y} = \mathbf{V}(:, j) \rightarrow \bm{y}$\;
    $\mathbf{L}^T \bm{x} = \bm{y} \rightarrow \bm{x}$\Comment*[r]{Solve two triangular systems of dimension $n_k$ at each iteration.}
    $\mathbf{W}(:, j) \gets \mathbf{D}^T \bm{x}$\;
}
\end{algorithm}

\subsection{Selection of the coarse nodes}\label{sec:selection_nodes}

We first observe that the number and location of the vertex and edge degrees of freedom described in \ref{dof:vertices} and \ref{dof:edge} are fixed a priori by the $\con{0}{}$-conformity requirement (Property \ref{prop:continuity}), together with the requirement that, for each edge $e \in \Eh$, the associated set of nodes be $k$-unisolvent, i.e., uniquely determine the polynomials in $\Poly[1]{k}{e}$ (Property \ref{prop:polynomial}).

The selection of the $n_{k-3}^2$ face coarse nodes from the set of face fine nodes should follow the same principles: it should depend only on the geometry of the face, and not on the polyhedron, to ensure the $\con{0}{}$-conformity requirement. Moreover, the selected set of face coarse nodes, indexed by $\mathcal{I}_{f}^{\mathcal{C}}$, must ensure that the set of face coarse nodes is $k$-unisolvent for the space of two-dimensional polynomials $\Poly[2]{k}{f}$ for every $f \in \Fh$. This condition also ensures that the matrices $\mathbf{D}^f$ have full row rank and, consequently, the applicability of Algorithm \ref{alg:kkt} to determine the face virtual weights. 

Finally, the selection of the $n_{k-4}^3$ coarse nodes, indexed by $\mathcal{I}_E^{\mathcal{C}}$, must be finalized to ensure the set of polyhedron coarse nodes is $k$-unisolvent for the space of three-dimensional polynomials $\Poly[3]{k}{E}$ for every $E \in \Th$. Also for this set, $k$-unisolvence ensures that the matrices $\mathbf{D}^E$ have full row rank and, consequently, the applicability of Algorithm \ref{alg:kkt} to determine the polyhedron virtual weights. 

Since $\Poly[3]{k}{E} \subset \mathbb{V}_k(E; \TE[E])$, there always exists a suitable choice of the subsets $\mathcal{I}_f^{\mathcal{C}}$ and $\mathcal{I}_E^{\mathcal{C}}$ that guarantees $k$-unisolvence. However, for intricate geometries, such as concave elements or elements with hanging nodes, identifying such internal points can be geometrically challenging and computationally demanding. To address this issue, the authors in \cite{NevaTeora2025} propose a heuristic selection strategy for the two-dimensional case, based on the following two principles:
\begin{itemize}
\item minimizing the alignment of the internal points;
\item ensuring a homogeneous spatial distribution of the internal points within the polygon.
\end{itemize}
The proposed strategy consistently produces a full-rank matrix $\mathbf{D}^f$ for all the tested cases. Since it relies solely on DOF index values, it is computationally inexpensive and can be readily extended to the three-dimensional setting.

An alternative strategy to the heuristic strategy to choose coarse nodes is presented next, and it will be called the QR-strategy. This QR-strategy is more computationally demanding than the heuristic approach presented in \cite{NevaTeora2025}, but relies on a more systematic procedure and is therefore more robust and less prone to failure. Like the heuristic approach, it avoids the geometric complexity typically associated with classical serendipity strategies, as it does not require the geometric identification of aligned entities or the definition of problem-dependent geometric tolerances. Such geometric criteria become increasingly difficult to define and apply as the geometric dimension increases, particularly in three-dimensional settings. In contrast, the proposed QR-strategy does not depend on the geometric dimension and can therefore be applied in any dimension.

We present this new strategy for the three-dimensional case, but it can be applied analogously to the 2D setting.

As for the heuristic strategy, this QR-strategy must be applied only for $k \geq 4$ in 3D to properly choose among the internal nodes, indexed by $\mathcal{I}_E$, the $n_{k-4}^3$ internal coarse nodes. Indeed, boundary degrees of freedom are dictated by conformity requirements.
Thus, let us define the matrix $\mathbf{D}^{B} \in \R^{n_k^3 \times \left(\# \mathcal{I}^{\mathcal{C}}_{v,e} + \sum_{f \in \Fh[E]} \# \mathcal{I}^{\mathcal{C}}_f\right)}$ as the matrix collecting the values of three-dimensional monomials at the boundary degrees of freedom that must be inserted among coarse nodes to guarantee $\con{0}{}$-conformity. Moreover, let $\mathbf{A}^{E} \in \R^{n^3_k \times \# \mathcal{I}_E}$ be the matrix collecting the evaluations of three-dimensional monomials at all internal nodes, indexed by $\mathcal{I}_E$.

The QR strategy proceeds as follows:
\begin{enumerate}
\item First, compute the QR factorization with column pivoting \cite{Engler1997} of the matrix $\mathbf{D}^{B}$ in order to identify a basis for its column space. Since the QR factorization with column pivoting is a rank-revealing factorization, this can be done as follows:
\begin{equation}
[\mathbf{Q}^{B}, \mathbf{R}^{B}, \mathbf{pivot}^{B}, \rm{rank}^{B}] = qr(\mathbf{D}^{B}),
\end{equation}
where $\mathbf{Q}^{B} \in \R^{n^3_k \times n^3_k}$ and $\mathbf{R}^{B} \in \R^{n^3_k \times \left(\# \mathcal{I}^{\mathcal{C}}{v,e} + \sum{f \in \Fh[E]} \# \mathcal{I}^{\mathcal{C}}f\right)}$ are the matrices resulting from the factorization of $\mathbf{D}^{B}$, $\mathbf{pivot}^{B} \in \R^{\# \mathcal{I}^{\mathcal{C}}{v,e} + \sum_{f \in \Fh[E]} \# \mathcal{I}^{\mathcal{C}}_f}$ is the vector containing the column permutation induced by the pivoting strategy, and $\rm{rank}^{B}$ denotes the numerical rank of $\mathbf{D}^{B}$ detected by the rank-revealing QR factorization.

The first $\rm{rank}^{B}$ columns of $\mathbf{Q}^{B}$, i.e., $\mathbf{Q}^{B}(:,1:\rm{rank}^{B})$, form an orthonormal basis for the column space of $\mathbf{D}^{B}$.

    \item Next, remove from the columns of $\mathbf{A}^{E}$ their components in the directions spanned by $\mathbf{Q}^{B}(:,1:\rm{rank}^{B})$. Specifically, define
    \begin{equation}
        \overline{\mathbf{A}}^{E} \coloneq \mathbf{A}^{E} - \mathbf{Q}^{B}(:, 1:\rm{rank}^{B}) \left(\mathbf{Q}^{B}(:, 1:\rm{rank}^{B}) \right)^T \mathbf{A}^{E}.
    \end{equation}
    \item Finally, identify the columns of $\mathbf{A}^{E}$ that are most linearly independent of both each other and the columns of $\mathbf{D}^{B}$ by applying a QR factorization with column pivoting to $\overline{\mathbf{A}}^{E}$: 
    \begin{equation}
        [\mathbf{Q}^{E}, \mathbf{R}^{E}, \mathbf{pivot}^{E}, \rm{rank}^{E}] = qr(\overline{\mathbf{A}}^{E}).
    \end{equation}
    We then define the set of internal bulk coarse nodes, $\mathcal{I}E^{\mathcal{C}}$, by selecting the first $n{k-4}^3$ column indices returned by the pivoting strategy, i.e., $\mathbf{pivot}^{E}(1:n_{k-4}^3)$. The virtual nodes $\virtualnodes[E]$ are identified by the remaining column indices returned by the QR factorization, namely, $\mathbf{pivot}^{E}(n_{k-4}^3:)$.
\end{enumerate}

We remark that the same strategy can also be used to pursue the stingy choice typical of the serendipity elements, i.e. not always adding $n_{k-4}^3$, but only the proper number $N^S$ such that $N^S + \rm{rank}^{B} = n_k^3$. Nonetheless, for the sake of simplicity, in this preliminary paper, we prefer to always fix the number of internal degrees of freedom, choosing to follow the lazy choice.

\begin{remark}
    We observe that, for high orders of the methods, the matrix $\mathbf{D}$ can become ill-conditioned. Nonetheless, different techniques are available in the literature to improve the condition number of matrix $\mathbf{D}$ \cite{Mascotto2018, Teora2024}.
\end{remark}

\section{The Zipped Finite Element Discrete Problem}\label{sec:discreteproblem}

For any integer $k \geq 1$, the zipped finite element space related to the tessellation $\Th$ can be written as
\begin{equation*}
    \VPh{k} \coloneq \{ v \in \con{0}{\overline{\Omega}} \cap \VP: \ v_{|E} \in \VPh[E]{k}\ \quad \forall E \in \Th \}. 
\end{equation*}

The basis functions $\{\varphi_i\}_{i=1}^{\Ndof}$, with $\Ndof = \dim \VPh{k}$, related to the global space $\VPh{k}$, can be obtained by using standard gluing techniques. The following results, whose proofs are trivial, hold true.

\begin{proposition}
The Z-FEM basis functions $\{\varphi_i\}_{i=1}^{\Ndof}$ satisfy Property \ref{prop:continuity}, i.e. they are continuous across the entire PDE domain $\Omega$. 
\end{proposition}

\begin{proposition}
   The Z-FEM basis functions $\{\varphi_i\}_{i=1}^{\Ndof}$ satisfy the partition-of-unity property, i.e. 
   \begin{equation*}
       \sum_{i = 1}^{\Ndof} \varphi_i(\xx) = 1 \quad \forall \xx \in \Omega.
   \end{equation*}
\end{proposition}

The Z-FEM discretization of Problem \eqref{eq:varproblem} reads as: \textit{Find $u_h \in \VPh{k}$ such that} 
\begin{equation}
    \bilin{u_h}{v_h} = \scal[\Omega]{f}{v_h} \quad \forall v_h \in \VPh{k}.
    \label{eq:discreteproblem}
\end{equation}

As in the standard FEM formulation, the existence and uniqueness of the solution of Problem \eqref{eq:discreteproblem} are inherited from the continuous framework. Moreover, it is easy to prove the following a priori estimates.
\begin{theorem}\label{prop:stimeapriori}
    Let us assume that the domain $\Omega$ is a convex polyhedron. Under Mesh Assumption \ref{ass:meshassumption}, let $u \in \sob{k+1}{\Omega}$, $k \geq 1$, be the solution of \eqref{eq:varproblem}, and let $u_h \in \VPh{k}$ be the solution of the discrete problem \eqref{eq:discreteproblem}. It holds:
    \begin{equation}
         \norm[\leb{2}{\Omega}]{u-u_h} \leq C_0h^{k+1} |u|_{\sob{k+1}{\Omega}}\quad \text{and}\quad  \norm[\sob{1}{\Omega}]{u-u_h} \leq C_{\nabla}h^{k} |u|_{\sob{k+1}{\Omega}},
         \label{eq:apriori2}
    \end{equation}
    for two positive constants $C_0$ and $C_{\nabla}$ independent of $h$.
\end{theorem}
\begin{proof}
The proof strictly follows that of Proposition 6 in \cite{NevaTeora2025}.
\end{proof}


\section{Numerical Experiments}
\label{sec:numericalexperiments}

In this section, we present some numerical results that aim to prove the polynomial reproduction property of Z-FEM and the theoretical estimates presented in the previous section. In these tests, for the ease of implementation, we always adopt the heuristic strategy on faces and the QR-strategy in the bulk of the element. Moreover, all these numerical experiments are performed using the C++ library \href{https://www.polydim.it/}{\texttt{PolyDiM}} \cite{Polydim}.

\begin{remark}[Note on implementation]
    If the QR-strategy is also applied on faces, all the internal coarse nodes associated with faces are computed once and for all, guaranteeing $\con{0}{}$-conformity independently of the particular polyhedron from which each face is seen. This ensures that the QR-strategy is not affected by the different accuracy with which face geometry may be computed depending on the starting polyhedron, while also speeding up the code. Only afterwards, the remaining internal polyhedron coarse nodes are computed.
\end{remark}

\subsection{Test 1: Patch test}

\begin{table}[!ht]
\centering
  \caption{Test 1: Polynomial approximation errors \eqref{eq:polapproxerrors} related to different polyhedra for different method orders.} 
  \label{tab:test1:singularvalues}
\resizebox{\textwidth}{!}{
  \begin{tabular}
      {ccc}
      \includegraphics[width=1in]{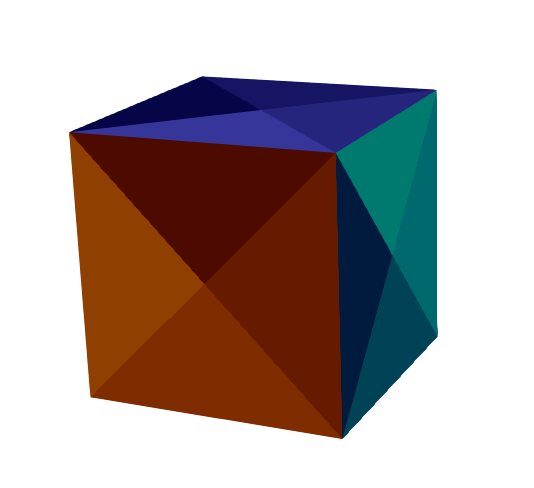} &
      \includegraphics[width=1in]{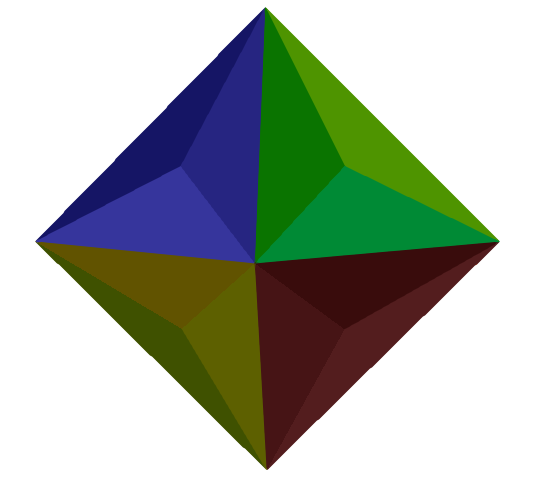} &
      \includegraphics[width=1in]{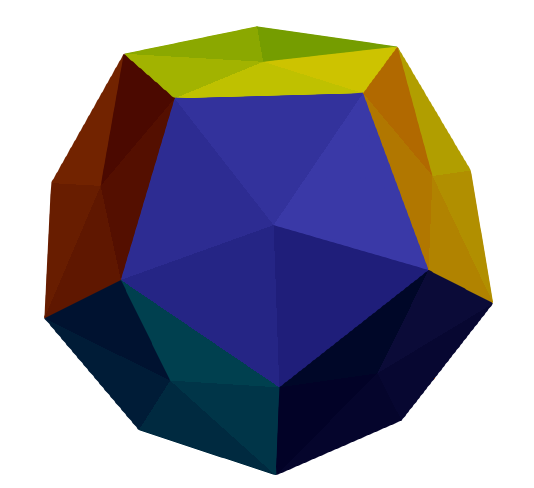}\\
      \multicolumn{1}{c}{Hexahedron} & \multicolumn{1}{c}{Octahedron} & \multicolumn{1}{c}{Dodecahedron} \\
      $\Nv[E] = 8,\ \Ne[E] = 12,\ \Nf[E] = 6$ & $\Nv[E] = 6,\ \Ne[E] = 12,\ \Nf[E] = 8$ & $\Nv[E] = 20,\ \Ne[E] = 30,\ \Nf[E] = 12$ \\
\begin{tabular}{ccc}
$k$ & $\mathrm{err}_{I,0}$ & $\mathrm{err}_{I,\nabla}$ \\
1 & 2.3172e-16 & 4.5862e-16 \\
2 & 5.4006e-15 & 1.8264e-14 \\
3 & 4.8046e-15 & 5.7472e-14 \\
4 & 3.5438e-14 & 5.1822e-13 \\
5 & 1.0818e-13 & 1.4237e-12   
\end{tabular} &
      \begin{tabular}{ccc}
$k$ & $\mathrm{err}_{I,0}$ & $\mathrm{err}_{I,\nabla}$ \\
1 & 2.0549e-16	& 1.1113e-15	\\
2 & 1.6568e-15	& 6.2162e-15	\\
3 & 1.5208e-15	& 1.1168e-14	\\
4 & 4.1500e-14	& 3.0786e-13	\\
5 & 6.7266e-14	& 9.3505e-13	
\end{tabular} & 
      \begin{tabular}{ccc}
$k$ & $\mathrm{err}_{I,0}$ & $\mathrm{err}_{I,\nabla}$ \\
1	& 4.7072e-16	& 8.9155e-16	\\
2	& 1.9983e-14	& 1.4892e-13	\\
3	& 1.1777e-13	& 5.3084e-13	\\
4	& 3.7154e-12	& 5.5668e-11	\\
5	& 1.4021e-12	& 2.8300e-11	
\end{tabular} \\
      \includegraphics[width=1in]{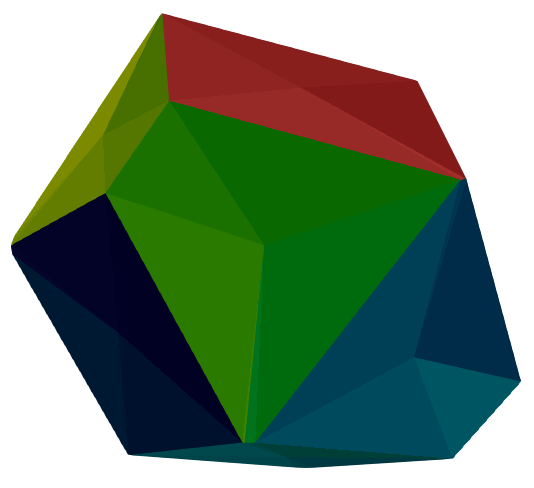} &
      \includegraphics[width=1in]{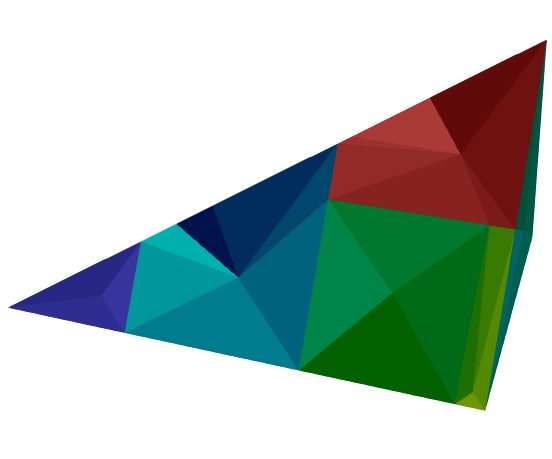} &
      \includegraphics[width=1in]{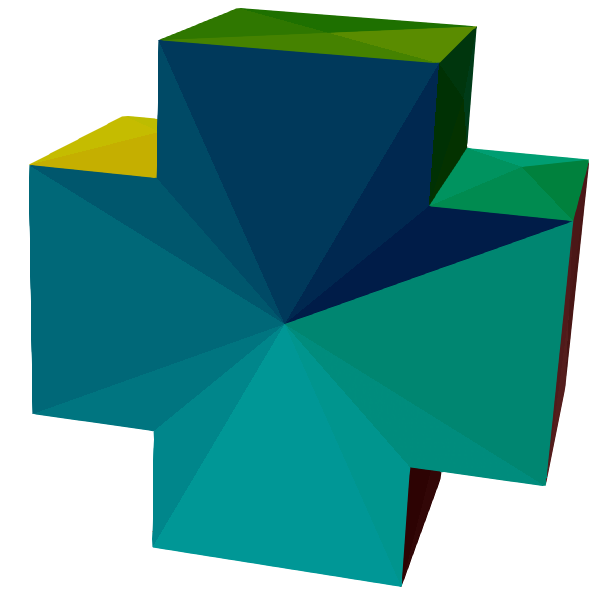}\\
      \multicolumn{1}{c}{Voronoi} & \multicolumn{1}{c}{Conformed} &\multicolumn{1}{c}{\begin{tabular}[c]{@{}c@{}}Cross\end{tabular}}  \\
      $\Nv[E] = 28,\ \Ne[E] = 42,\ \Nf[E] = 16$ & 
      $\Nv[E] = 18,\ \Ne[E] = 26,\ \Nf[E] = 10$ 
      & $\Nv[E] = 24,\ \Ne[E] = 36,\ \Nf[E] = 14$  \\
      \begin{tabular}{ccc}
$k$ & $\mathrm{err}_{I,0}$ & $\mathrm{err}_{I,\nabla}$ \\
1&2.6003e-17	&9.6614e-16	\\
2&4.8637e-16	&5.0216e-14	\\
3&1.6996e-15	&1.9487e-13	\\
4&2.1091e-13	&2.1827e-11	\\
5&2.4532e-13	&6.8499e-11
\end{tabular} &
      \begin{tabular}{ccc}
$k$ & $\mathrm{err}_{I,0}$ & $\mathrm{err}_{I,\nabla}$ \\
1&1.1087e-17	&2.1595e-15	\\
2&8.0430e-17	&5.1851e-15	\\
3&1.1119e-15	&7.7572e-14	\\
4&2.4173e-14	&2.2749e-12	\\
5&2.4827e-13	&2.5746e-11	
\end{tabular} &
      \begin{tabular}{ccc}
$k$ & $\mathrm{err}_{I,0}$ & $\mathrm{err}_{I,\nabla}$ \\
1	&4.2149e-16&	1.1398e-15\\	
2	&6.0926e-15&	2.4533e-14\\	
3	&1.2734e-14&	7.8666e-14\\	
4	&8.6864e-13&	3.9074e-12\\	
5	&2.4456e-13&	4.0051e-12	
\end{tabular} \\
  \end{tabular}
  }
\end{table}

The first test we propose aims to show how accurately the local basis functions are able to reproduce polynomials. For this purpose, we consider a set of different polyhedra and, on each polyhedron, we compute the following polynomial approximation errors
\begin{equation}
    \mathrm{err}_{I,0} = \max_{\alpha =1, \dots, n_k^3} \norm[\leb{2}{E}]{m_{\alpha}^3 - \mathcal{I}_k^E m_{\alpha}^3},\quad \mathrm{err}_{I,\nabla} = \max_{\alpha =1, \dots, n_k^3} \norm[\leb{2}{E}]{\nabla m_{\alpha}^3 - \nabla \mathcal{I}_k^E m_{\alpha}^3},
    \label{eq:polapproxerrors}
\end{equation}
where the interpolator $\mathcal{I}_k^E: \con{0}{\overline{E}} \cap \sob{1}{E} \to \VPh[E]{k}$ is defined as
\begin{equation*}
    \mathcal{I}_k^E v = \sum_{i \in \coarsenodes[E]} v (\xx_i^E) \varphi_i \quad \forall v \in \con{0}{\overline{E}} \cap \sob{1}{E}.
\end{equation*}

In Table~\ref{tab:test1:singularvalues}, we show six examples extracted from the tested cases. For each one, we draw in the top part the considered polyhedron, with colors that highlight the underlying tetrahedralization, and report the number of vertices, edges, and faces. The tested examples include polyhedra sampled from different classes: three Platonic solids, namely the hexahedron, the octahedron, and the icosahedron, an element extracted from a random Voronoi mesh, one element extracted by a \textit{conformed} mesh, created in the context of Discrete Fracture and Matrix applications \cite{Teora2024}, which is available in the PolyDiM mesh dataset, and which is characterized by aligned edges and faces, and a star-shaped concave element, representing a cross-shaped element.

Along with the polyhedra pictures, we report in Table \ref{tab:test1:singularvalues} the polynomial reproducibility errors $\mathrm{err}_{I,0}$ and $\mathrm{err}_{I,\nabla}$, defined in \eqref{eq:polapproxerrors}, for polynomial orders $k = 1, \dots, 5$. In all tested cases, both errors are found to be very small, with the greatest values on the order of $1.0e\rm{-}11$, demonstrating that the proposed shape functions can approximate polynomials with very high accuracy. Moreover, we emphasize that the matrices $\mathbf{D}$ related to these experiments have always full rank, thereby confirming that the assumptions of Proposition~\ref{prop:optproblemsolution} are satisfied for both faces and bulk matrices.

\subsection{Test 2: Convergence rates}

In this example, we solve the Problem \eqref{eq:varproblem} over the unit cube $\Omega = [0,1]^3$, by setting
\begin{equation}
    \D(\xx) = e^{x + y + z} \quad \gamma(\xx) = xyz \quad \forall \xx = (x,y,z) \in \Omega.
\end{equation}

We choose the forcing term and the Dirichlet boundary condition in accordance with the exact solution:
\begin{equation*}
    u(\xx) = \sin(\pi x) \sin(\pi y) \sin(\pi z).
\end{equation*}

In this test case, we assess the performance of the proposed method by solving this problem on different families of meshes. For each family of meshes, we compute the following standard errors:
\begin{equation}
    \mathrm{err}_0 = \norm[\leb{2}{\Omega}]{u - u_h},\qquad \mathrm{err}_{\nabla} = \norm[\leb{2}{\Omega}]{\nabla u - \nabla u_h},
    \label{eq:errors}
\end{equation}
as the mesh size $h$ decreases and for $k=1,\dots, 5$.

\begin{figure}[!ht]
    \centering
    \begin{subfigure}{0.4\textwidth}
        \includegraphics[width=1\linewidth]{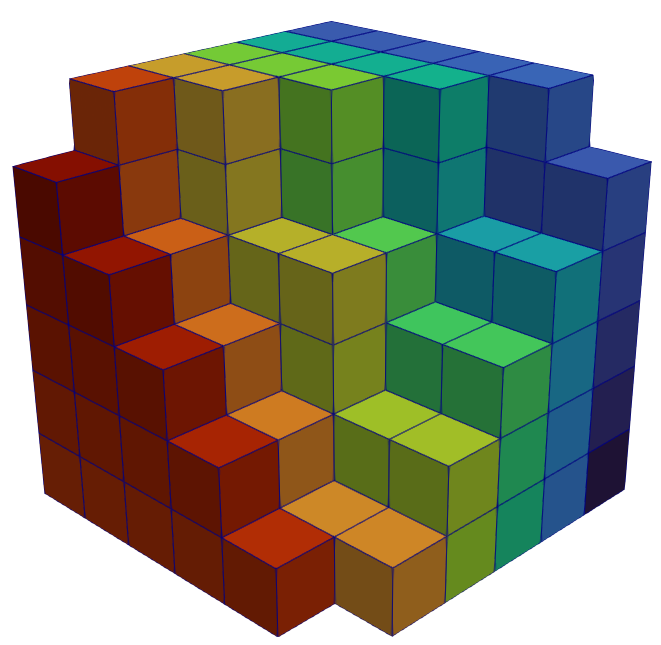}
        \caption{Cubic.}
    \end{subfigure}
    \begin{subfigure}{0.4\textwidth}
        \includegraphics[width=1\linewidth]{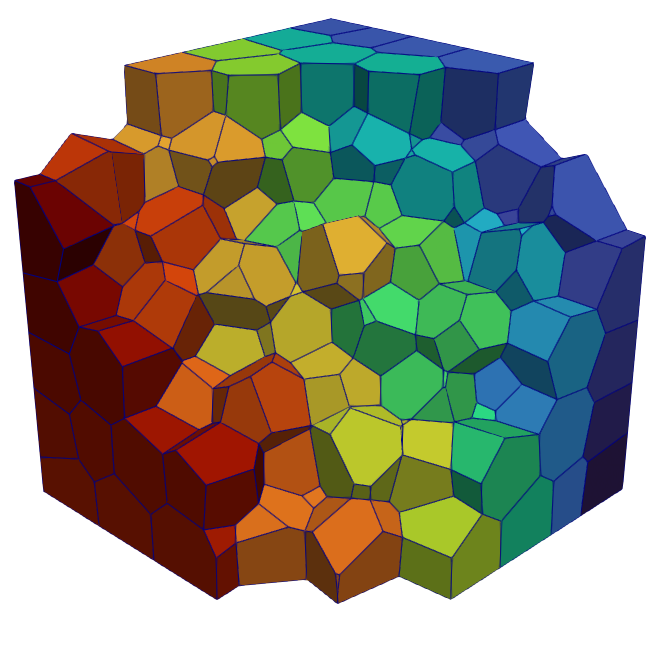}
        \caption{Random Voronoi.}
    \end{subfigure}
    \begin{subfigure}{0.4\textwidth}
        \includegraphics[width=1\linewidth]{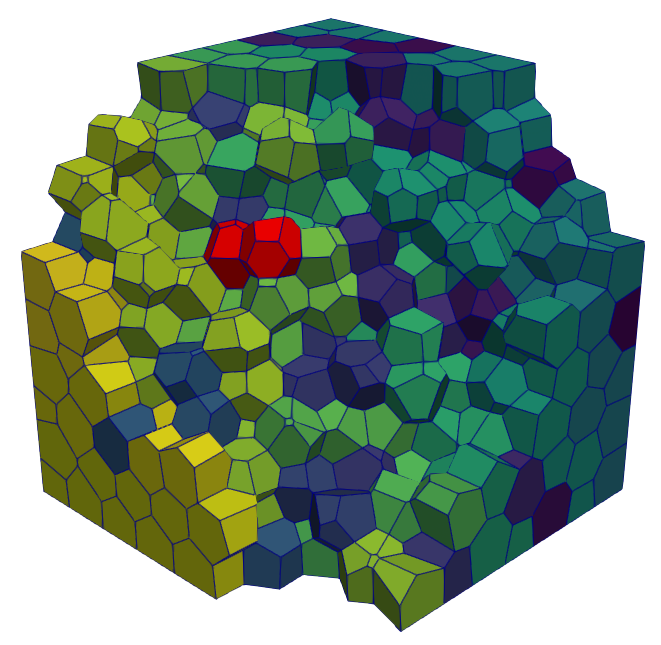}
        \caption{Agglomerated Random Voronoi.\\ In red, element 239.}
    \end{subfigure}
    \begin{subfigure}{0.4\textwidth}
        \includegraphics[width=1\linewidth]{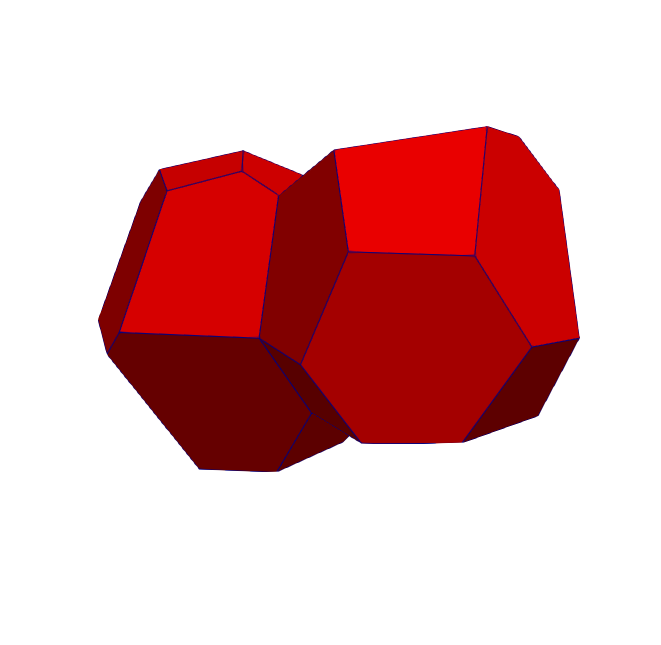}
        \caption{Zoom-in on element 239.}
    \end{subfigure}
    \caption{A cross-section of the latest refinements for each mesh family.}
    \label{fig:test2:mesh}
\end{figure}

Specifically, we consider the following three families of meshes, each made up of 5 refinements:
\begin{itemize}
\item A family of Cubic meshes;
\item A family of Random Voronoi meshes generated using \cite{Voro};
\item A family of Agglomerated Random Voronoi meshes. The agglomeration is performed by preserving the star-shaped properties for all the elements.
\end{itemize}
Figure~\ref{fig:test2:mesh} shows a cross-section of the finest mesh for each family.

\begin{figure}[!ht]
    \centering
    \hspace{-30pt}\begin{subfigure}{0.42\textwidth}
        \includegraphics[width=1.2\linewidth]{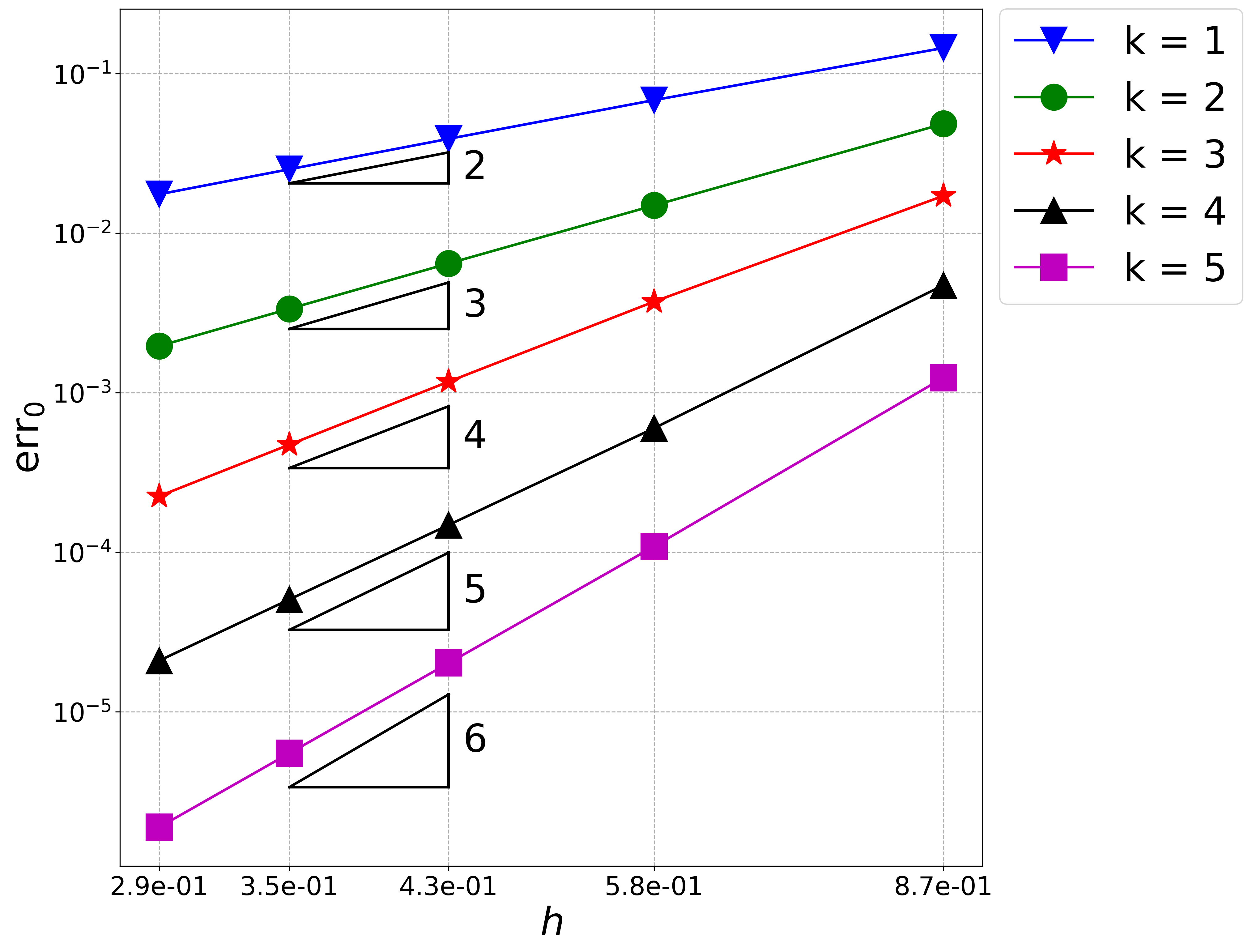}
        \caption{}
    \end{subfigure}\hspace{-12pt}
    \begin{subfigure}{0.42\textwidth}
        \includegraphics[width=1.2\linewidth]{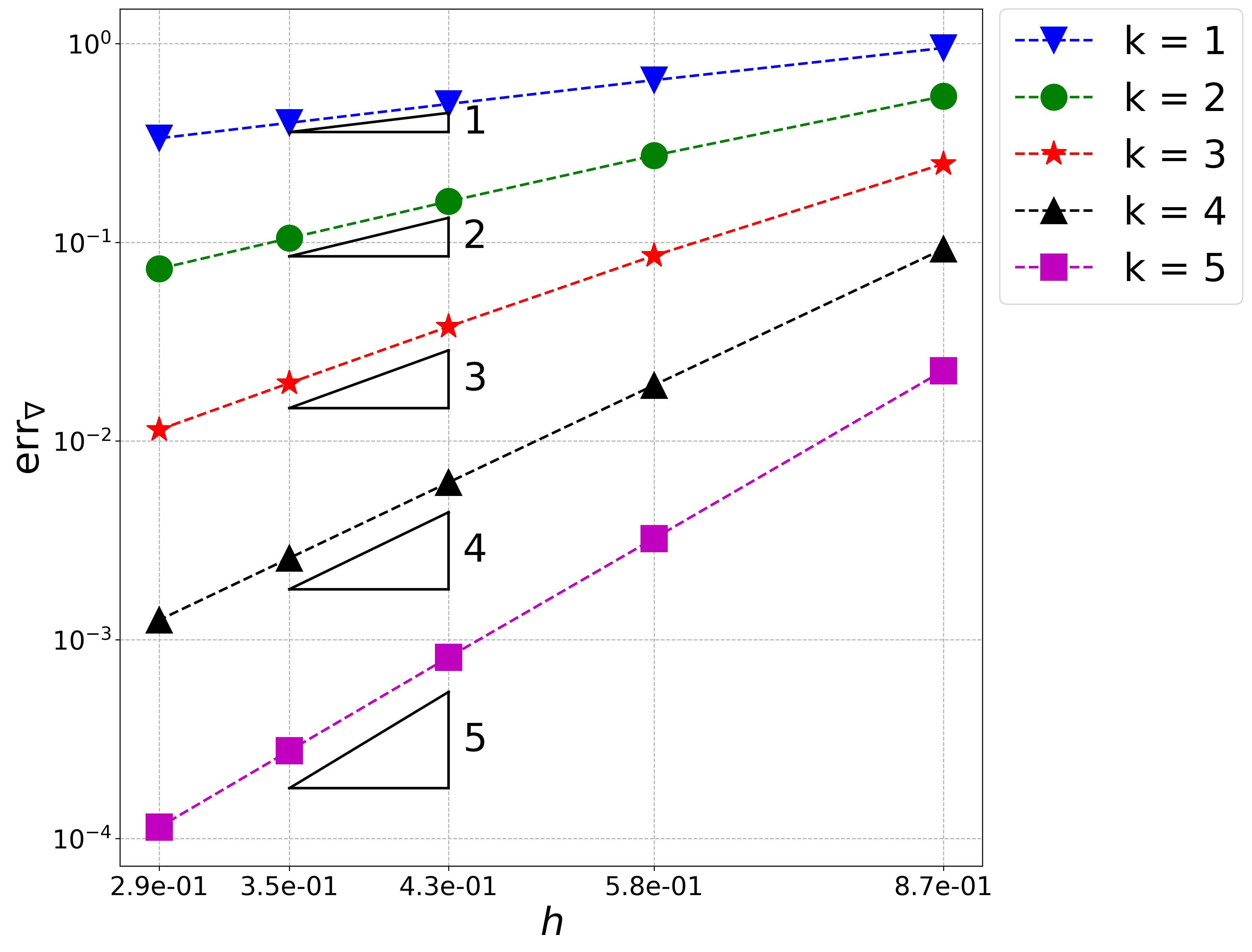}
        \caption{}
    \end{subfigure}\\
    \hspace{-30pt}\begin{subfigure}{0.42\textwidth}
        \includegraphics[width=1.2\linewidth]{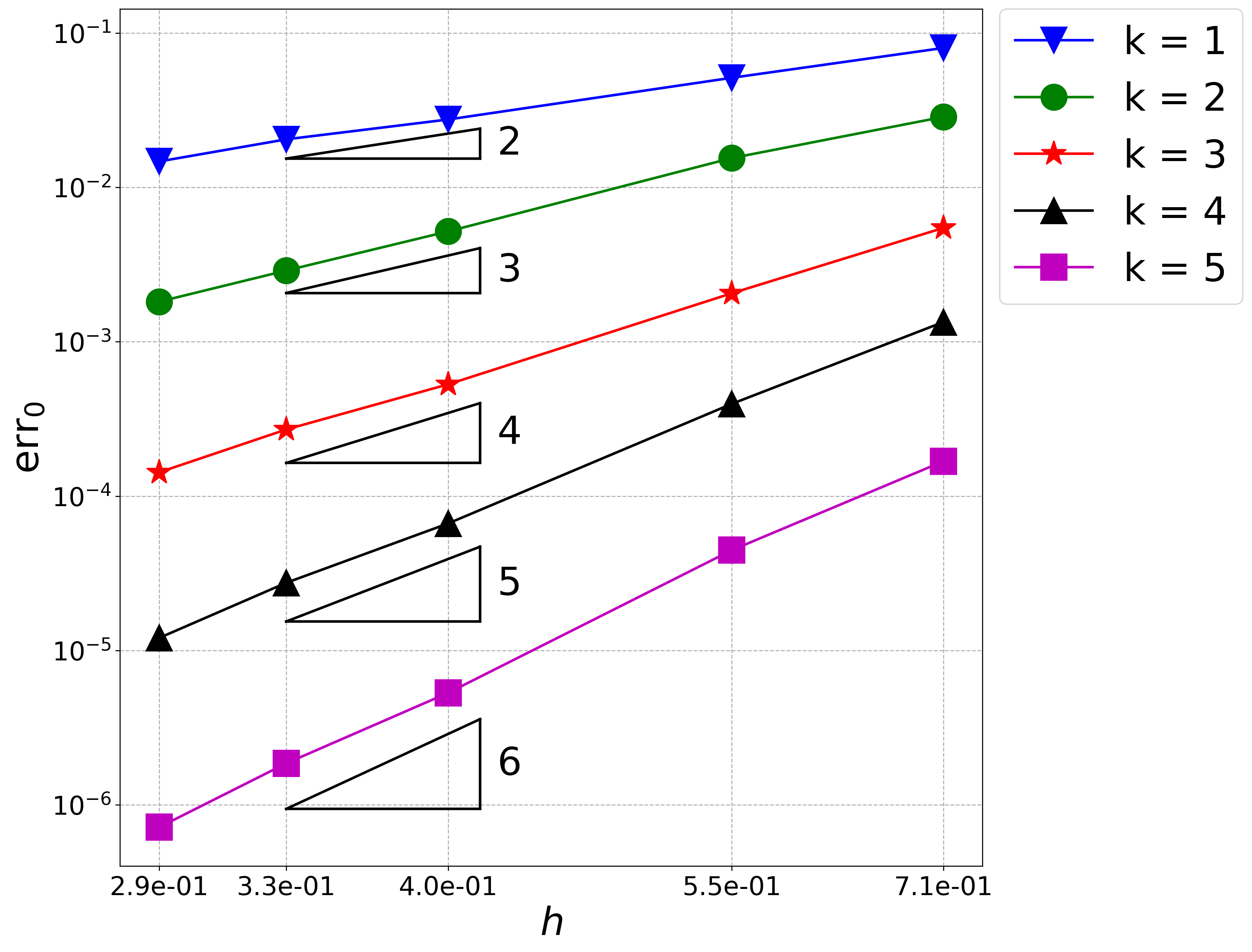}
        \caption{}
    \end{subfigure}\hspace{-12pt}
    \begin{subfigure}{0.42\textwidth}
        \includegraphics[width=1.2\linewidth]{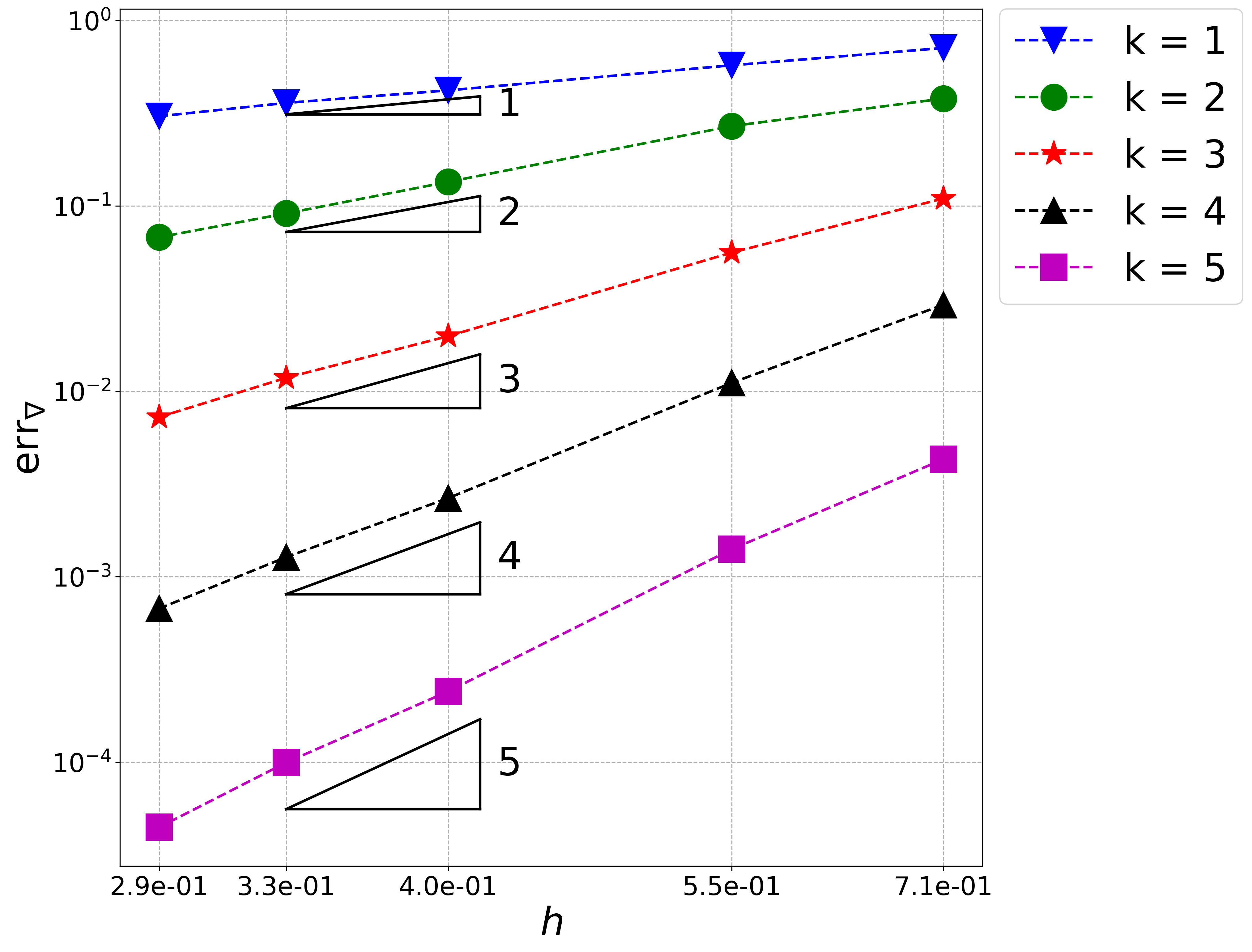}
        \caption{}
    \end{subfigure}\\
    \hspace{-30pt}\begin{subfigure}{0.42\textwidth}
        \includegraphics[width=1.2\linewidth]{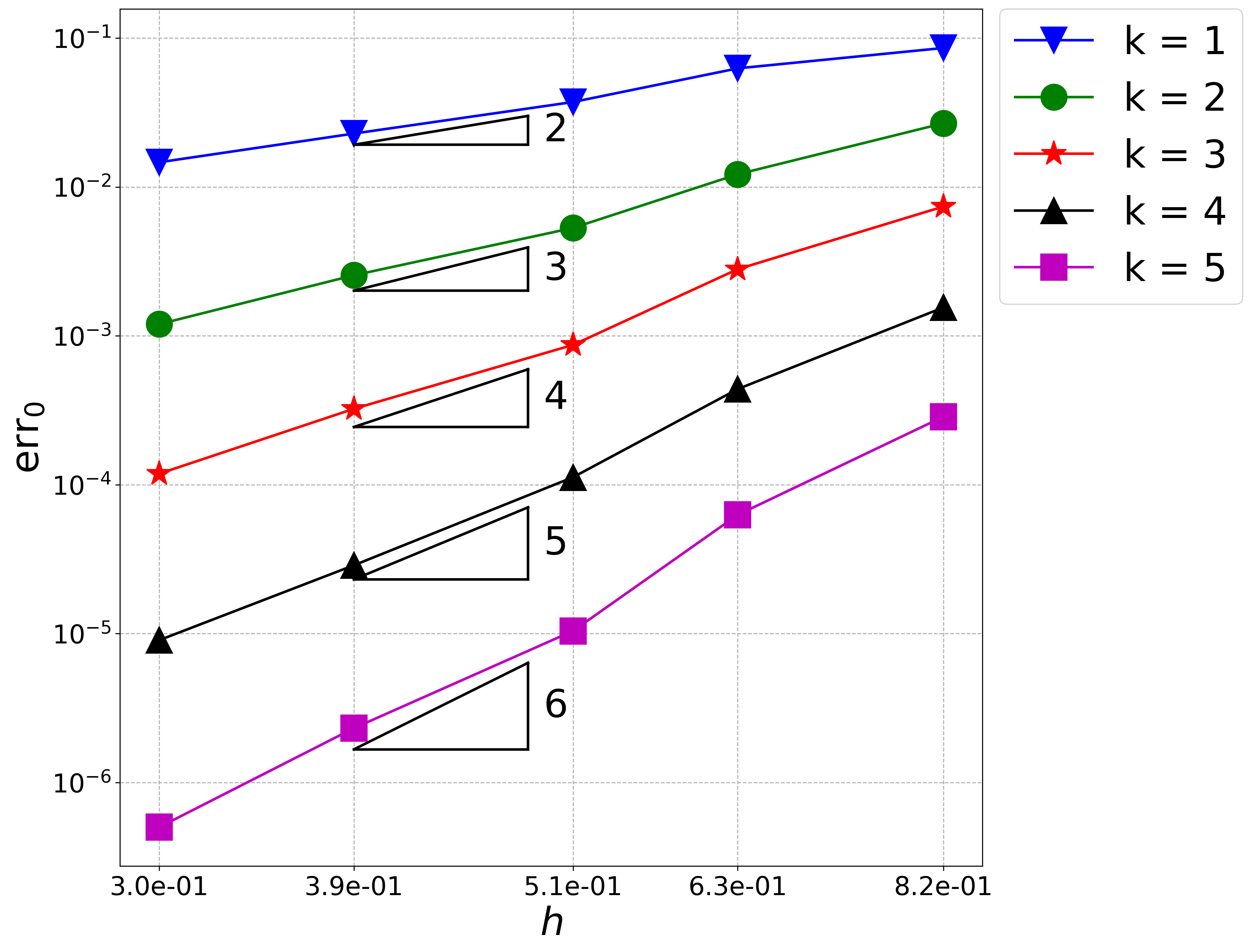}
        \caption{}
    \end{subfigure}\hspace{-12pt}
    \begin{subfigure}{0.42\textwidth}
        \includegraphics[width=1.2\linewidth]{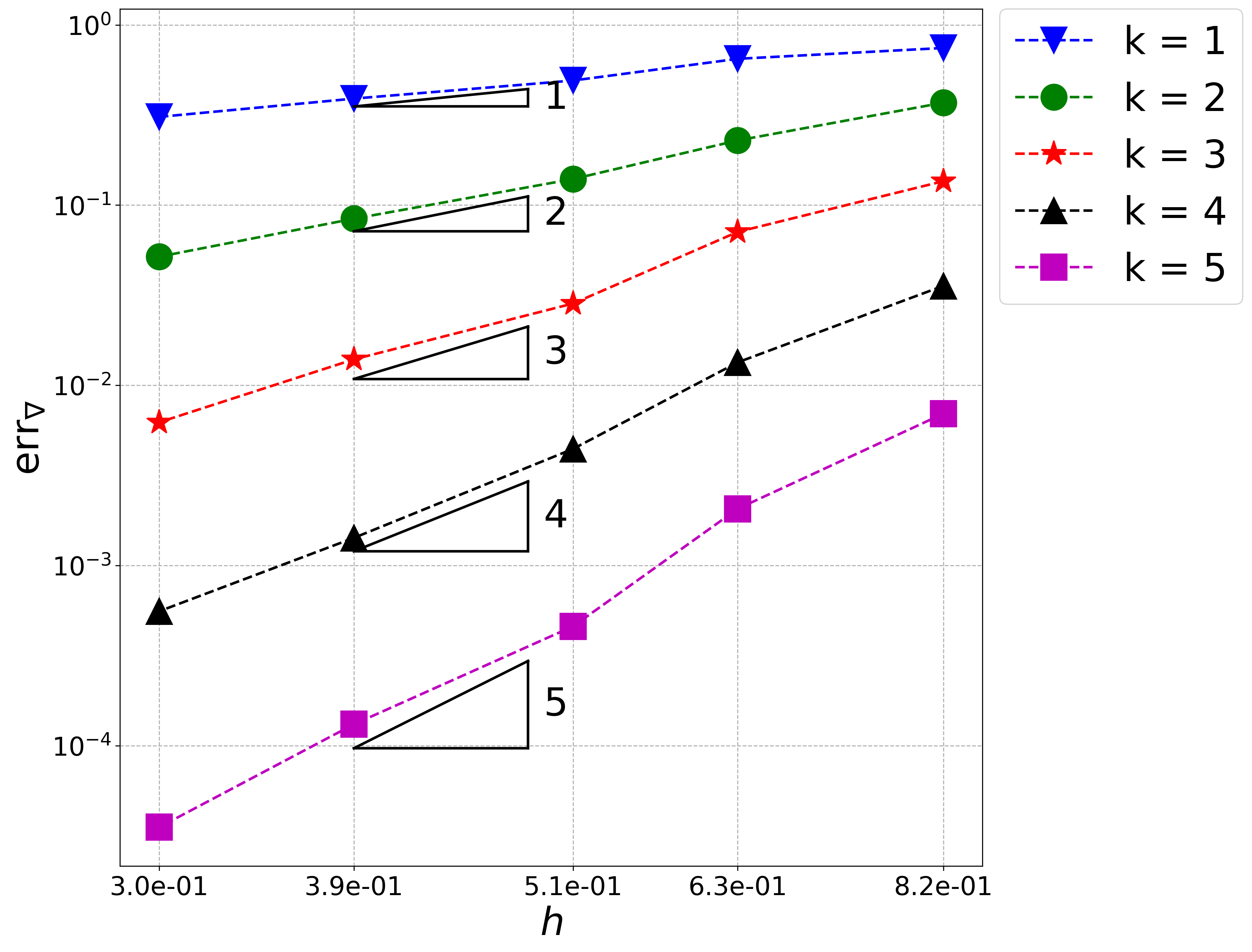}
        \caption{}
    \end{subfigure}
    \caption{Behavior of errors \eqref{eq:errors} for as $h$ decreases, for $k=1,\dots,5$. First row: Cubic. Second row: Random Voronoi. Third row: Agglomerated Random Voronoi.}
    \label{fig:test2:errors}
\end{figure}

Figure~\ref{fig:test2:errors} illustrates the behavior of the errors defined in \eqref{eq:errors} as the mesh size $h$ decreases, for $k=1,\dots,5$ and for each family of meshes. Table~\ref{tab:test2:rate}, instead, reports the corresponding empirical orders of convergence, denoted by $\alpha_0$ and $\alpha_{\nabla}$ for the $\leb{2}{}$- and $\sob{1}{}$-errors, respectively. Both errors decay with an empirical order of convergence that matches the theoretical one, for each order of the method and independently of the considered family of meshes. 

Finally, we consider the DOFs reduction obtained by applying the Z-FEM procedure instead of a standard FEM on the sub-tetrahedralization. The percentage representing this DOFs reduction, denoted here by $\mathrm{gain}$, is defined as
\begin{equation}
    \rm{gain} = \frac{\Ndof(\TE) - \Ndof}{\Ndof(\TE)} 100,
    \label{eq:gain}
\end{equation}
where $\Ndof(\TE)$ represents the number of degrees of freedom associated with the FEM defined over the union of sub-tetrahedralizations. These values are reported in Table~\ref{tab:test2:rate} for the finest mesh of each family and for each order. The higher the order $k$, the greater is the saving in terms of DOFs.

\begin{table}[!ht]
\centering
\caption{Empirical orders of convergence $\alpha_0$ and $\alpha_{\nabla}$ related to the $\leb{2}{}$- and $\sob{1}{}$-errors and the gain \eqref{eq:gain} of DOFs related to the finest refinement for each mesh family. }
\label{tab:test2:rate}
\begin{tabular}{@{}l|ccc|ccc|ccc@{}}
    & \multicolumn{3}{c|}{\textbf{Cubic}}          & \multicolumn{3}{c|}{\textbf{Random Voronoi}} & \multicolumn{3}{c}{\textbf{Agglomerated Random Voronoi}} \\ \midrule
$k$ & $\alpha_0$ & $\alpha_{\nabla}$ & $\rm{gain}$ & $\alpha_0$ & $\alpha_{\nabla}$ & $\rm{gain}$ & $\alpha_0$     & $\alpha_{\nabla}$     & $\rm{gain}$     \\ \midrule
1   & 1.92       & 0.95              & 73.92       & 1.85       & 0.93              & 57.81       & 1.81           & 0.91                  & 54.99           \\
2   & 2.92       & 1.82              & 85.14       & 3.10       & 1.95              & 81.44       & 3.10           & 1.97                  & 80.35           \\
3   & 3.95       & 2.81              & 88.98       & 4.02       & 3.01              & 88.19       & 4.16           & 3.11                  & 87.69           \\
4   & 4.92       & 3.91              & 90.83       & 5.22       & 4.19              & 91.19       & 5.20           & 4.22                  & 90.98           \\
5   & 5.9        & 4.81              & 91.91       & 6.08       & 5.08              & 92.85       & 6.38           & 5.31                  & 92.82           \\ \bottomrule
\end{tabular}
\end{table}

\section{Conclusion}\label{sec:conclusion}

The paper presents the three-dimensional extension of the high-order Zipped-Finite Element Method. In particular, it illustrates techniques for the appropriate selection of nodal degrees of freedom within polyhedral elements and for the construction of basis functions that are $\con{0}{}$-conformed and are able to reproduce polynomials up to the method order. Numerical experiments are conducted on both convex and concave meshes, and the results validate the theoretical findings.

Future research will focus on the application of the method to more realistic problems, where Z-FEM may benefit from combining its FEM-derived theoretical framework with the advantages offered by polytopal methods. 

\section*{Acknowledgments}

The author S.B. kindly acknowledges partial financial support provided by European Union through project Next Generation EU, M4C2, PRIN 2022 PNRR project P2022BH5CB\_001 ``Polyhedral Galerkin methods for engineering applications to improve disaster risk forecast and management: stabilization-free operator-preserving methods and optimal stabilization methods''. The author G.T. kindly acknowledges the financial support provided by INdAM-GNCS Project ``Metodi numerici politopali stabilization-free e neural-based per problemi accoppiati e non lineari'' (CUP: E53C25002010001). The authors S.B. and G.T. kindly acknowledge the financial support provided by the European Union under the ``BANDO AGGREGAZIONI R\&S - TRANSIZIONE ECOLOGICA'' project ``OPTIMOLD, Ottimizzazione Integrata per lo Stampaggio Plastico Sostenibile'' (CUP: B19J25000450007).

\bibliographystyle{IEEEtranDOI}
\bibliography{biblio.bib}

\end{document}